\documentclass[12pt]{amsart}

\usepackage{amssymb,latexsym}
\usepackage{amsmath,amsfonts,amsthm,amscd,amsxtra,enumerate}
\usepackage[all]{xy}
\usepackage{amsmath,latexsym,amssymb, enumerate, amsthm, epsfig, hyperref} 

\usepackage[margin=1in]{geometry}
\usepackage{epsfig}

\newtheorem{theorem}{Theorem}[section]
\newtheorem{lemma}[theorem]{Lemma}
\newtheorem{corollary}[theorem]{Corollary}
\newtheorem{definition}[theorem]{Definition}
\newtheorem{example}[theorem]{Example}
\newtheorem{remark}[theorem]{Remark}
\newtheorem{proposition}[theorem]{Proposition}

\newtheorem{question}[theorem]{Question}
\newtheorem{setup}[theorem]{Setup}

\newcommand{\p}{\mathbf{P}}
\newcommand{\sk}{{\ensuremath{\sf k }}}
\newcommand{\m}{\ensuremath{\mathfrak m}}
\newcommand{\ZZ}{\ensuremath{\mathbb{Z}}}
\newcommand{\QQ}{\ensuremath{\mathbb{Q}}}

\newcommand{\Y}{\ensuremath{\mathbf{Y}}}
\newcommand{\Z}{\ensuremath{\mathbf{Z}}}
\newcommand{\y}{\ensuremath{\mathbf{y}}}
\newcommand{\z}{\ensuremath{\mathbf{z}}}

\DeclareMathOperator{\ann}{ann}
\DeclareMathOperator{\Tor}{Tor}
\DeclareMathOperator{\soc}{soc}
\DeclareMathOperator{\im}{im}
\DeclareMathOperator{\edim}{edim}
\DeclareMathOperator{\type}{type}
\DeclareMathOperator{\gr}{gr}

\begin{document}

\title{\textbf{Associated Graded Rings and Connected Sums}}

\author[H. Ananthnarayan]{H. Ananthnarayan}
\address{Department of Mathematics, I.I.T. Bombay, Powai, Mumbai 400076.}
\email{ananth@math.iitb.ac.in}

\author[E. Celikbas]{Ela Celikbas}
\address{Department of Mathematics, West Virginia University, Morgantown, WV 26506.}
\email{ela.celikbas@math.wvu.edu}

\author[Jai Laxmi]{Jai Laxmi}
\address{Department of Mathematics, I.I.T. Bombay, Powai, Mumbai 400076.}
\email{jailaxmi@math.iitb.ac.in}

\author[Z. Yang]{Zheng Yang}
\address{Department of Mathematics, Miami University, Oxford, OH 45056.}
\email{yangz15@miamioh.edu}



\thanks{Z. Yang was partially supported by NSF grant DMS-1103176.}

\begin{abstract}
In 2012, Ananthnarayan, Avramov and Moore
gave a new construction of Gorenstein rings from two Gorenstein local rings, called their connected sum.
In this article, we investigate conditions
on the associated graded ring of a Gorenstein Artin local ring $Q$, which force it to be a connected sum over its residue field. In particular, we recover some results regarding short, and stretched, Gorenstein Artin rings. Finally, using these decompositions, we obtain results about the rationality of the Poincar$\acute{\text{e}}$ series of $Q$.
\end{abstract}

\maketitle

\section*{Introduction}

Given Gorenstein Artin local rings $R$ and $S$ with common residue field $\sk$, and the natural surjective maps $R \overset{\pi_R}\longrightarrow \sk \overset{\pi_S}\longleftarrow S$, a connected sum is an appropriate quotient of the fibre product (or pullback) $R\times_\sk S = \{(a,b) \in R \times S: \pi_R(a) = \pi_S(b)\}$.
Lescot, in \cite{Le}, proves that the connected sum is also Gorenstein Artin. A more general version of a connected sum of two Gorenstein local rings having the same dimension is defined by Ananthnarayan, Avramov and Moore in \cite{AAM}. They prove that this new construction is also a Gorenstein local ring of the same dimension.  

It is natural to ask when a given Gorenstein Artin local ring $Q$ can be decomposed as a connected sum over its residue field $\sk$.
In the equicharacteristic case, this question was studied by Smith and Stong in \cite[Section 4]{SS} from a geometric point of view for {\it projective bundle ideals}, and also by using
inverse systems by Buczy\'nska et al. in \cite{BBKT} using polynomials that are
{\it direct sums}, corresponding to {\it apolar} Gorenstein algebras. The study of intrinsic properties of the ring, under which it is indecomposable as a connected sum, 
and other conditions on its defining ideal characterizing decomposability, are given in \cite{ACJY}. 

In \cite[Proposition 4.5]{ACJY}, it is shown that if $Q$ is a standard graded $\sk$-algebra, the decomposability of $Q$ as a connected sum over $\sk$ can be characterized in terms of the quotient of $Q$ by its socle. This leads us to look at the non-graded case, and in particular to the study of the associated graded ring. The main focus of this paper is to understand the connections of properties of the associated graded ring of $Q$ with the decomposability of $Q$ as a connected sum over $\sk$.

In Section \ref{GrIndec}, we study the associated graded ring $G$ of a Gorenstein Artin local ring $Q$ which can be decomposed as a connected sum over its residue field $\sk$. In particular, in Proposition \ref{Prop2}, we see that if the components of $Q$ have different Loewy lengths, then $G$ can be decomposed as a fibre product over $\sk$. As seen in examples in Section \ref{GrFP}, the converse is not necessarily true, i.e., if $G$ decomposes as a fibre product over $\sk$, then $Q$ need not decompose as a connected sum over $\sk$. 

In light of this, one can ask when the converse is true. In order to study this, assuming $G \simeq A \times_\sk B$, we impose further conditions on $A$ and $B$ in Setup \ref{setup}, and investigate properties of $Q$ in this case in Theorem \ref{EXJ}. The motivation for Setup \ref{setup} comes from the following:
In \cite{Sa}, Sally proves a structure theorem for stretched Gorenstein rings and in \cite{ER}, Elias and
Rossi give a similar structure theorem for short Gorenstein $\sk$-algebras with some assumptions on the residue field.
In particular, these structure theorems show
that the ring $Q$ can be written as a connected sum of a graded Gorenstein Artin ring $R$ with the same Loewy length
as $Q$, and a Gorenstein Artin ring $S$ with Loewy length less than three.
In either case, using a construction of Iarrobino we see in Proposition \ref{maincor} that $G \simeq A \times_\sk B$, where $A$ is graded Gorenstein, and $B$ has Loewy length 1. 

In Section \ref{Gr}, assuming Setup \ref{setup}, we look for conditions on $G \simeq A \times_\sk B$ that force $Q$ to be a connected sum. 
Using the characterization of connected sums (Theorem \ref{Characterization}), and the results from Section \ref{GrFP}, 
we show that $Q$ is a connected sum, in general, when the Loewy length of $B$ is one, and in some special cases, when the Loewy length of $B$ is two. This gives us results regarding the Poincar$\acute{\text{e}}$ series of $Q$.
(See Theorem \ref{StructureTheorem}, Corollary \ref{Loewy(B)=2}, and their corollaries in Section \ref{SS}). Example \ref{4.16} shows that if the Loewy length of $B$ is at least 2, then $Q$ need not be a connected sum.

In Section \ref{SS}, we use Theorem \ref{StructureTheorem} to give applications to short and stretched
Gorenstein Artin local rings. In particular, we show that these rings, when they are not graded, are
non-trivial connected sums, and derive some consequences without any restrictions on the residue field. We also identify some Gorenstein Artin local rings with rational Poincar$\acute{\text{e}}$ series.

The first section contains results regarding the main tools used in the rest of the paper. It also contains results about fibre products and connected sums,
including a characterization of connected sums in terms of a minimal generating set of maximal ideals (Theorem \ref{Characterization}).

\section{Preliminaries}\label{Preliminaries}

\subsection{Notation}\hfill{}

\begin{enumerate}[{\rm a)}]
\item  If $T$ is a local ring and $M$ is a $T$-module, $\lambda(M)$ and $\mu(M)$ respectively denote the {\it length} and the
{\it minimal number of generators} of $M$ as a $T$-module.

\item Let $(T,\m,\sk)$ be an Artinian local ring. Then $\edim(T)$
denotes the {\it embedding dimension} of $T$ which is equal to $\mu(\m)$.
The {\it socle} of $T$ is $\soc(T) = \ann_T(\m)$. Moreover, the {\it type} of $T$ is $\type(T) = \dim_{\sk}(\soc(T))$,
and the \emph{Loewy length} of $T$ is $\ell\ell(T) = \max\{n: \m^n \neq 0\}$.\footnote{If $T$ is also Gorenstein,
its Loewy length is also referred to as {\it socle degree} in the literature.}

Observe that $T$ is not a field if and only if $\ell\ell(T) \geq 1$. Furthermore, if $T$ is Gorenstein, then $\soc(T) \subset \m^2$ if and only if $\ell\ell(T) \geq 2$.

\item If $\sk$ is a field, a {\it graded $\sk$-algebra} $G$ is a graded ring $G = \oplus_{i \geq 0} G_i$ with
$G_0 = \sk$. It has a unique homogeneous maximal ideal, $\m_G = \oplus_{i \geq 1} G_i$. We say $G$ is {\it standard graded} if $\m_G$ is generated by $G_1$.

\item  For positive integers $m$ and $n$, $\Y$ and $\Z$ denote the sets of indeterminates $\{Y_1, \ldots, Y_m\}$ and
$\{Z_1,\ldots,Z_n\}$ respectively, and $\Y \cdot \Z$ denotes $\{Y_iZ_j: 1 \leq i \leq m, 1 \leq j \leq n\}$.
\end{enumerate}

\subsection{Associated graded rings}
\begin{definition} Let $(T,\m,\sk)$ be a Noetherian local ring.\begin{enumerate}[{\rm a)}]
\item The {\it graded ring associated to the maximal ideal} $\m$ of
$P$, denoted $\gr_{\m}(T)$, {\rm (}or simply $\gr(T)${\rm)}, is defined as $\gr(T) \simeq \oplus_{i=0}^{\infty} (\m^i/\m^{i+1})$.

\item We define the {\it Hilbert function} of $T$ as $H_T(i) = \dim_{\sk}(\m^i/\m^{i+1})$ for $i \geq 0$.

\item If  $T$ is an Artinian ring with $\ell\ell(T) = s$, then we write the Hilbert function of $T$
as $H_T=(H_T(0),\ldots,H_T(s))$.

Furthermore, if $T$ is Gorenstein, we say that $T$ is \emph{stretched} if
$H_T =  (1,h,1,\cdots,1)$, i.e., $\m_T^2$ is principal and $T$ is \emph{short} if $\m^4_T = 0$.
\end{enumerate}
\end{definition}

\begin{remark}\label{GR}{\rm With notation as above, let $G = \gr(T)$ and
$G_{\geq n} = \bigoplus_{i=n}^{\infty} \m^i/\m^{i+1}$ for $n \geq 0$.

\begin{enumerate}[a)]
\item For each $n \geq 0$, $G_{\geq n}$ is the $n$th power of the homogeneous maximal ideal $\m_G$ of $G$ and a minimal
generating set of $G_{\geq n}$ lifts to a minimal generating set of $\m^n$.

In particular, if $T$ is Artinian local, then so is $G$. Furthermore, $\lambda(T)  = \lambda(G)$ and $\ell\ell(T) = \ell\ell(G)$.

\item For each $x \in T \setminus \{0\}$, there exists a unique $i \in \ZZ$ such that $x \in \m^i \setminus \m^{i+1}$. The
{\it initial form} of $x$ is the element $x^* \in G$ of degree $i$, that is, the image of $x$ in $\m^i/\m^{i+1}$.

\item For an ideal $K $ of $T$, $K^*$ denotes the ideal of $G$ defined by $\langle x^*: x \in K \rangle$.
Note that, if $R \simeq T/K$, then $\gr(R) \simeq G/K^*$.

\item If $T$ is a regular local ring, $\m = \langle x_1,\ldots,x_d\rangle$, then $x_1^*, \ldots, x_d^* \in G$ are algebraically independent, and hence $G$ is isomorphic to a polynomial ring over $\sk$ in $d$ variables.
\end{enumerate}
}\end{remark}

\begin{remark}[Iarrobino's Construction(\cite{Ia})]\label{Iarrobino} {\rm
Let $(Q,\m_Q,\sk)$ be a Gorenstein Artin local ring with $\ell\ell(Q)=s$, and let $G = \gr(Q)$ be its associated graded ring.
Iarrobino showed that
$$C = \bigoplus_{i \geq 0}  \frac{(0:_Q \m_Q^{s-i})\cap \m_Q^i}{(0:_Q \m_Q^{s-i})\cap \m_Q^{i+1}}$$
is an ideal in $G$. He also proved that $Q_0 = G/C$ is a graded Gorenstein quotient of $G$ with
$\deg(\soc(Q_0)) = s$.

Note that $H_{Q_0}(i) = H_G(i)$ for $i \geq s-1$ since $C_i = 0$ for $i \geq s-1$.
}\end{remark}

\subsection{Cohen Presentations and Poincar$\mathbf{\acute{\text{e}}}$ Series}

\begin{definition}
{\rm Let $T$ be a local ring. We say that $\widetilde T / I_T$ is a {\it Cohen presentation} of $T$ if $(\widetilde T, \m_{\widetilde T}, \sk)$ is a complete regular local ring and $I_T \subset \m_{\widetilde T}^2$ is an ideal in $\widetilde T$ such that $T \simeq \widetilde T / I_T$.
}\end{definition}

\begin{remark}\hfill{}{\rm

\begin{enumerate}[a)]
\item By Cohen's Structure Theorem, every complete Noetherian local ring has a Cohen presentation.
\item Let $(\widetilde T, \m_{\widetilde T}, \sk)$ be a complete regular local ring, $I$ be an ideal in $\widetilde T$. Set $T =  \widetilde T / I$. Then $\widetilde T / I$ is a Cohen presentation of $T$ if and only if $\edim(\widetilde T) = \edim(T)$.
\end{enumerate}
}\end{remark}

\begin{definition}\label{PSdef}{\rm For a local ring $(T, \m, \sk)$, the {\it Poincar$\acute{\text{e}}$ series} of $T$, is the formal power series
$$\p^T(t)=\sum_{i\geq 0} \beta_i^Tt^i, \;\;\; \text{with}\;\; \beta_i^T=\dim_{\sk}\left( \Tor_i^T(\sk, \sk)\right).$$
}\end{definition}

\begin{remark}[Minimal Number of Generators]\label{PS2}{\rm
Let $(T,\m,\sk)$ be a complete Noetherian local ring with $\edim(T) = d$, and $T = \widetilde T/I_T$ be a Cohen presentation.
By \cite[Corollary 7.1.5]{Av}, we have
$$\beta_1^T =d\quad\text{ and }
\quad\mu(I_T) =\beta_2^T - \binom{\beta_1^T}{2} = \beta_2^T - \binom{d}{2}.$$
}\end{remark}

Next we list some properties of the Poincar$\acute{\text{e}}$ series of a Gorenstein Artin local ring.

\begin{remark}[Poincar$\mathbf{\acute{\text{e}}}$ Series of Gorenstein Rings]\label{PS} \hfill{}\\{\rm
Let $(T,\m,\sk)$ be a Gorenstein Artin local ring and $\overline T$ represent the quotient $T/\soc(T)$.

\begin{enumerate}[a)]
\item If $\edim(T) \leq 4$, then $\p^T(t)$ is rational. See \cite[Theorem 6.4]{AKM}.

\item If $\edim(T) \geq 2$, then
$\p^{T}(t) = \p^{\overline T}(t) + t^{-2}$,  by \cite[Theorem 2]{LA}.

\item If $\ell\ell(T) = 2$ and $\edim(T) = n$, then $\p^{T}(t) = (1- t)^{-1}$ if $n = 1$. If $n \geq 2$, then $\p^{T}(t) = (1 - nt + t^2)^{-1}$. In particular, $\p^T(t)$ is a rational function of $t$.
\end{enumerate}
}\end{remark}

\subsection{Connected Sums}

We first define the fibre product of local rings $(R,\m_R,\sk)$ and $(S,\m_S,\sk)$ over $\sk$. If neither of them is a field, then their fibre product is not Gorenstein. Hence, we define an appropriate quotient called a connected sum which is Gorenstein. For more details about the contents of this subsection, see \cite[Section 2]{AAM}, \cite[Chapter 4]{HAthesis}, and \cite{ACJY}.
 
\begin{definition}\label{FPdef}
Let $(R,\m_R,\sk)$ and $(S,\m_S,\sk)$ be local rings. The {\it fibre product} $R$ and $S$ over $\sk$ is the ring
$R\times_{\sk} S = \{(a,b) \in R \times S: \pi_R(a) = \pi_S(b)\}$,
where $\pi_R$ and $\pi_S$ are the natural projections from $R$ and $S$ respectively onto $\sk$.
\end{definition}

\begin{remark}\label{FP} {\rm
A Noetherian local ring $(P,\m_P,\sk)$ is decomposable as a fibre product over $\sk$ if and only if $\m_P$ is minimally generated by $\{y_1\ldots,y_m,z_1,\ldots,z_n: m, n \geq 1\}$, where $y_iz_j = 0$. In this case, we have $P \simeq R\times_\sk S$, where $R \simeq P/\langle \z\rangle$ and $S \simeq P/\langle \y\rangle$. 

On the other hand, with $R$ and $S$ as in Definition \ref{FPdef}, if $P = R\times_{\sk} S$, then, by identifying $\m_R$ with $\{(a,0): a \in \m_R\}$ and $\m_S$ with $\{(0,b): b \in \m_S\}$, we see that $P$ is a local ring with maximal ideal $\m_P = \m_R \times \m_S$. Hence, we see that $\edim(P) = \edim(R) + \edim(S)$, $\soc(P)=\soc(R)\oplus \soc(S)$ and by \cite[(1.0.3)]{AAM}, $\gr(P) \simeq \gr(R) \times_\sk \gr(S)$. This implies that for $i \geq 1$, we have $H_P(i) = H_R(i) + H_S(i)$. Furthermore, $R \simeq P/\m_S$ and $S \simeq P/\m_R$.
}\end{remark}

\begin{definition}\label{CSdef}
Let $(R,\m_R,\sk)$ and $(S,\m_S,\sk)$ be Gorenstein Artin local rings different from $\sk$.
Let $\soc(R) = \langle\delta_R\rangle$, $\soc(S) = \langle\delta_S\rangle$. Identifying $\delta_R$ with
$(\delta_R, 0)$ and $\delta_S$ with $(0, \delta_S)$, a {\it connected sum} of $R$ and $S$ over $\sk$,
denoted $R\#_\sk S$, is the ring $R\#_\sk S = (R\times_\sk S)/\langle\delta_R - \delta_S\rangle$.
\end{definition}

A connected sum of $R$ and $S$ is Gorenstein, see \cite[Proposition 4.4]{Le}, or \cite[Theorem 2.8]{AAM}. 
Connected sums of $R$ and $S$ also depend on
the generators of the socle $\delta_R$ and $\delta_S$ chosen.
In \cite[Example 3.1]{AAM} it is shown that for $R = \QQ[Y]/\langle Y^3\rangle$ and $S = \QQ[Z]/\langle Z^3\rangle$ , the connected sums  $Q_1 = (R \times_\sk S) / \langle y^2 - z^2\rangle$ and $Q_2 = (R \times_\sk S) /\langle y^2 - 5 z^2\rangle$ of are not isomorphic as rings.
Finally, observe that every Gorenstein Artin local ring $(Q,\m_Q,\sk)$ can be decomposed trivially over $\sk$ as $Q \simeq Q \#_\sk \frac{\sk[Z]}{\langle Z^2 \rangle}$. Hence the following is defined in \cite{ACJY}.

\begin{definition}{\rm
Let $(Q,\m,\sk)$ be a Gorenstein Artin local ring. We say that $Q$ {\it decomposes as a connected sum}
over $\sk$ if there exist Gorenstein Artin local rings $R$ and $S$ such that
$Q \simeq R \#_\sk S$ and $R \not\simeq Q \not\simeq S$.
In this case, we call $R$ and $S$ the {\it components} in a connected sum decomposition of $Q$, and say that
$Q \simeq R \#_\sk S$ is a non-trivial decomposition.

If $Q$ cannot be decomposed as a connected sum over $\sk$,
we say that $Q$ is {\it indecomposable} as a connected sum over $\sk$.
}\end{definition}

The following theorem gives a characterization of connected sums over $\sk$ which was proved in \cite{ACJY}; parts (a) through (d) have been established in \cite{AAM} and \cite{ACJY}. In this paper we prove part (e) in Proposition \ref{Prop2}.

\begin{theorem}[Connected Sums over $\sk$]\label{Characterization}
Let $(Q,\m_Q,\sk)$ be a Gorenstein Artin local ring with $\ell\ell(Q) \geq 1$. Then $Q$ can be decomposed nontrivially as a connected sum over $\sk$ if and only if $\m_Q = \langle y_1,\ldots, y_m, z_1,\ldots, z_n \rangle$, $m$, $n \geq 1$, with $\y \cdot \z = 0$.

If $Q$ decomposes as a connected sum over $\sk$, then we can write $Q \simeq R \#_\sk S$, such that
$R =\widetilde Q/J_R$, $S = \widetilde Q/J_S$ with
$J_R = (I_Q \cap \langle \Y \rangle) + \langle \Z \rangle$, and
$J_S = (I_Q \cap \langle\Z\rangle) + \langle \Y \rangle$, where $\widetilde Q/I_Q$ is a Cohen presentation of $Q$. Furthermore, the following hold:

\begin{enumerate}[{\rm a)}]
\item $\lambda(Q) = \lambda(R) + \lambda(S) - 2$ and $\edim(Q) = \edim(R) + \edim(S)$.
\vskip 2pt

\item For $0 < i < \min\{\ell\ell(S),\ell\ell(R)\}$, we have $H_Q(i) = H_R(i) + H_S(i) \leq \binom{m + n-2+i}{i}+1$.
\vskip 5pt

\item $I_Q = (J_R \cap \langle \Y \rangle) + (J_S \cap \langle \Z \rangle) + \langle \Y \cdot \Z\rangle+ \langle \Delta_R-  \Delta_S\rangle$, where
$\Delta_R \in \langle\Y\rangle$ and $\Delta_S \in \langle\Z\rangle$ are such that their respective
images $\delta_R \in R$ and $\delta_S \in S$ generate the respective socles.
\vskip 2pt

\item $\mu(I_Q) = \mu\left(J_R/\langle \Z \rangle\right) + \mu\left(J_S/\langle \Y \rangle\right) + mn + \phi_{m,n}$, and
{\large $\frac{1}{\p^{Q}(t)} =  \frac{1}{\p^R(t)} + \frac{1}{\p^S(t)}$} $- 1 -\phi_{m,n} t^2$
\vskip 2pt
where $\phi_{m,n}$ is $1$ for $m$, $n \geq 2$, $\phi_{1,1} = -1$, and $\phi_{m,n} = 0$ otherwise.
\vskip 2pt
\item If $\ell\ell(R) > \ell\ell(S) \geq 2$, then $\gr(Q) \simeq \gr(R) \times_\sk \gr\left(S/\soc(S)\right)$.
\end{enumerate}\end{theorem}

Parts (b) and (d) of Theorem~\ref{Characterization} give us some conditions to determine when a Gorenstein ring is indecomposable as a connected sum.

\begin{remark}\label{indecomp.rmk} {\rm Let $(Q,\m_Q,\sk)$ be a Gorenstein Artin local ring with $\edim(Q)=d$. Then $Q$ is indecomposable over $\sk$ if one of the following condition holds:
\begin{enumerate}[{\rm a)}]
\item $Q$ is a complete intersection ring and $d\geq 3$ (see \cite[Theorem 3.6]{ACJY}).

\item $H_Q(2)\geq \binom{d}{2}+2$. (see \cite[Theorem 3.9]{ACJY}).
\end{enumerate}
 }
\end{remark}

\section{Associated Graded Rings and Indecomposibility}\label{GrIndec}
The following basic property of the associated graded ring is used to obtain conditions for indecomposibility.

\begin{proposition}\label{Prop2}
For Gorenstein Artin local rings $(R,\m_R,\sk)$ and $(S,\m_S,\sk)$, let $Q = R\#_\sk S$. If $\ell\ell(R)\neq \ell\ell(S)$, then the associated graded ring of $Q$ is a fibre product over $\sk$.\par
Moreover, if $R$ and $S$ are standard graded $\sk$-algebras with $\ell\ell(R),\ell\ell(S)\geq 2$, then the nontrivial connected sum $R\#_{\sk}S$ is not standard graded.
\end{proposition}
\begin{proof}
Let $P=R\times_{\sk}S$, $\soc(R)=\langle\delta_R\rangle$ and $\soc(S)=\langle\delta_S\rangle$. Since $Q\simeq P/\langle \delta_R - u\delta_S \rangle$ for some unit $u$ in $S$, we have $\gr(Q) \simeq \gr(P)/\langle \delta_R-u\delta_S \rangle^*$ by Remark \ref{GR}(c). Recall that, by Remark \ref{FP}, we have $\gr(P) \simeq \gr(R) \times_\sk \gr(S)$.

Without loss of generality, we may assume that $\ell\ell(R) > \ell\ell(S)$.
Hence $\langle \delta_R- u\delta_S \rangle^* = \langle u\delta_S \rangle^*$. Thus we see that
$\gr(Q) \simeq (\gr(R)\times_\sk \gr(S))/\langle u \delta_S \rangle^*\simeq \gr(R) \times_\sk \gr(S/\langle \delta_S \rangle)$.

Now, note that $\gr(R) \neq \sk \neq
\gr(S/\langle \delta_S \rangle)$ for standard graded $\sk$-algebras $R$ and $S$ with $\ell\ell(R) > \ell\ell(S) \geq 2$. Since $\gr(Q)$  is decomposable nontrivially as a fibre product over $\sk$,
it is not Gorenstein. Thus $Q \not\simeq \gr(Q)$, and hence $Q$ is not standard graded.
\end{proof}

The condition $\ell\ell(R) \neq \ell\ell(S)$ is necessary in the above proposition. To see this, observe that if $R$ and $S$ are standard graded $\sk$-algebras with $\ell\ell(R) = \ell\ell(S) \geq 2$, then $Q = R \#_\sk S$ is a non-trivial connected sum. However, since $\ell\ell(R) = \ell\ell(S)$, $Q$ is standard graded, and therefore, $\gr(Q) \simeq Q$ is Gorenstein. Hence, $\gr(Q)$ is indecomposable as a fibre product over $\sk$, which leads us to the following:

\begin{question}\label{Q2.2}
Let $(Q,\m_Q,\sk)$ be a Gorenstein Artin local ring. Does the indecomposibility of $\gr(Q)$ as a fibre product over $\sk$ imply $\gr(Q)\simeq Q$, or at least force $\gr(Q)$ to be Gorenstein?
\end{question}

The following example shows that Question \ref{Q2.2} has a negative answer.

\begin{example}{\rm
Let $R = \QQ[Y_1,Y_2]/\langle Y_1^2Y_2, Y_1^3 - Y_2^2\rangle$,
$S = \QQ[Z]/\langle Z^5 \rangle$ and $Q \simeq R \#_\QQ S$.
Note that $\gr(Q) \simeq \QQ[Y_1,Y_2,Z]/\langle Y_1Z, Y_2Z, Y_1^2Y_2,Y_2^2,Y_1^4 - Z^4\rangle$. Set $G=\gr(Q)$. We see that $G$ is not Gorenstein since $\soc(G) = \langle Y_1Y_2,Z^4 \rangle$.
Note that $R$ is not standard graded, and $\ell\ell(R) = \ell\ell(S) = 4$. We now show that $G$ is indecomposable as a fibre product over $\QQ$.

Suppose $G \simeq A \times_\QQ B$ is a non-trivial fibre product, for some $\QQ$-algebras $A$ and $B$.
Since $\edim(G) = 3$, by Remark \ref{FP}, we may assume that $\edim(A) = 2$ and $\edim(B) =1$.
Moreover, $2=\type(G)=\type(A)+\type(B)$ forces $A$ and $B$ to be Gorenstein, and the fact that $\soc(G) = \soc(A) \oplus \soc(B)$ implies that $\ell\ell(A) = 4$ and $\ell\ell(B) = 2$ or vice versa.

If $\ell\ell(B) = 4$, then $H_B(i)=1$ for $i\leq 4$. Since $H_G = (1, 3, 3, 2, 1)$, we get $H_A=(1, 2, 2, 1)$ by Remark~\ref{FP}. This implies that $\ell\ell(A) = 3$ which cannot happen. Thus we must have $\ell\ell(A) = 4$ and $\ell\ell(B) = 2$.

Thus, if $G \simeq A \times_\QQ B$ is a non-trivial fibre product, we may assume that $A$ and $B$ are Gorenstein,
$\edim(A) = 2$, $\ell\ell(A) = 4$, and $B \simeq \QQ[V]/\langle V^3\rangle$. Therefore, we can write
$A \simeq \QQ[U_1,U_2]/\langle f_1,f_2\rangle$ where $U_1$, $U_2$ and $V$ are indeterminates over $\QQ$ and $$G \simeq \QQ[U_1,U_2,V]/\langle U_1V,U_2V,V^3,f_1,f_2 \rangle.$$

Let lower case letters denote the respective images of the
indeterminates in $G$ and $$\phi: \QQ[U_1,U_2,V]/\langle U_1V,U_2V,V^3,f_1,f_2 \rangle \longrightarrow
\QQ[Y_1,Y_2,Z]/\langle Y_1Z, Y_2Z, Y_1^2Y_2,Y_2^2,Y_1^4 - Z^4\rangle$$ be an isomorphism.
Write $\phi(u_1) = a_{11}y_1 + a_{12}y_2 + a_{13}z$, $\phi(u_2) = a_{21}y_1 + a_{22}y_2 + a_{23}z$ and
$\phi(v) = a_{31}y_1 + a_{32}y_2 + a_{33}z$. The fact that $\phi(v^3) = 0$ forces $a_{31} = 0 = a_{33}$. Hence $a_{32} \neq 0$ and
therefore, $a_{11} = 0 = a_{21}$ since $\phi(u_1v) = 0 = \phi(u_2v)$. This gives us a contradiction as $a_{i1} = 0$ for all $i$
implies that $y_1 \not\in \im(\phi)$. Hence $G$ cannot be written as a non-trivial fibre product over $\QQ$.}
\end{example}

We record our observations from this section in the next remark:

\begin{remark}\label{GrFP}{\rm Let $(Q,\m_Q,\sk)$ be a Gorenstein Artin local ring, and $G = \gr(Q)$.\\
a) If $G$ is indecomposable as a fibre product over $\sk$. Then exactly one of the two occurs:\\
{\rm (i)} $Q \simeq R \#_\sk S$, where $R$ and $S$ are Gorenstein Artin local rings with $\ell\ell(R) = \ell\ell(S) \geq 2$, or\\
{\rm (ii)} $Q$ is indecomposable as a connected sum over $\sk$.\\
b) If $G$ is Gorenstein, then either (i) $Q$ is standard graded, (i.e., $Q \simeq G$), or (ii) $Q$ is indecomposable as a connected sum over $\sk$. This follows from (a), since, if $G$ is Gorenstein, then it is indecomposable as a fibre product over $\sk$. \\
c) The two examples $Q = \QQ[y,z]/\langle yz, y^2 - z^2 \rangle \simeq \QQ[y]/\langle y^3 \rangle \#_{\QQ} \QQ[z]/\langle z^3 \rangle$, and $Q = \QQ[x]/\langle x^3 \rangle$ respectively show that both, (i) and (ii), in (a), and (b) are possible scenarios.
}\end{remark}

\section{Associated Graded Rings as Fibre Products}\label{GrFP}

Given a Gorenstein Artin local ring $Q$, with $\gr(Q) = G$, we identify conditions on the $G$ that are necessary for $Q$ to decompose as a connected sum. One such condition arises out of Proposition \ref{Prop2}, namely that, with some assumptions, $G$ is a fibre product. This leads to the following:

\begin{question}
Let $Q$ be a Gorenstein Artin local ring such that $\gr(Q)$ is decomposable as a fibre product over $\sk$. Is $Q$ decomposable as a connected sum over $\sk$?
\end{question}

Example \ref{FPex1} shows that the above question need not have a positive answer in general.

\begin{example}\label{FPex1}{\rm
Let $Q=\sk[X,Y,Z]/I$ where $I=\langle X^{3}-YZ,Y^3-XZ,Z^2\rangle$. Then $Q$ is a complete intersection with $\edim(Q) = 3$, and hence $Q$ is indecomposable as a connected sum over $\sk$ by Remark~\ref{indecomp.rmk}(a). However, $\gr(Q)\simeq \sk[X,Y]/\langle X^{4}-Y^4, X^2Y^3,X^3Y^2\rangle\times_{\sk}\sk[Z]/\langle Z^2\rangle.$
}\end{example}

In light of the above example, we impose a further restriction on $\gr(Q)$, and assume the following setup.

\begin{setup}\label{setup}
Let $(Q,\m_Q,\sk)$ be a Gorenstein Artin local ring with $G = \gr(Q)$, and $\ell\ell(Q) = s$. Assume $G \simeq A\times_{\sk}B$, where $A$ is graded Gorenstein, and $B$ is graded. Let $\ell\ell(A) = s$, and $\ell\ell(B) = k$ with $s > k + 1$.
\end{setup}

The following remark has some basic observations, used (often without reference) throughout this section. \begin{remark}\label{FPrmk}{\rm With notations as in the above setup, let $\m_A$, $\m_B$ and $\m_G$ be the respective maximal ideals of $A$, $B$ and $G$ respectively.\\
a) Let $K_1$ and $K_2$ be ideals in $Q$. Since $Q$ is Gorenstein, $(0:_QK_1)$ can be identified with a canonical module of $Q/K_1$, and hence $\lambda(Q/K_1) = \lambda (0:_QK_1)$.\\ Moreover, $(0:_Q (0:_Q K_1)) = K_1$, and $(0:_Q (K_1\cap K_2)) = (0:_QK_1) + (0:_Q K_2)$. The equality $(0:_Q(K_1+K_2) = (0:_QK_1) \cap (0:_Q K_2)$ holds without the Gorenstein assumption. \\
b) The above properties also hold for ideals in $A$. Furthermore, since $A$ is graded Gorenstein, and $\ell\ell(A) = s$, we have $0:_A \m_A^{i+1} = \m_A^{s-i}$ for $0 \leq i \leq s-1$.\\
c) By Remark \ref{FP}, we have a natural projection $\pi\colon G \rightarrow A$ with $\ker(\pi) = \m_B$.
Moreover, $\ker(\pi) \cap (\m_G)^{k+1} =0$ since $\m_B^{k+1} = 0$. Thus $\m_G^i = \m_A^i$, for $i = k+1,\ldots,s$. Finally, $(0:_G\m_A^i) = (0:_A\m_A^i) + \m_B = \m_A^{s-i+1} + \m_B$ for each $i$. 
}\end{remark}

We begin with some observations that hold under the above setup.

\begin{lemma}\label{stabl}
With notation as in Setup \ref{setup}, we have the following:
\begin{enumerate}[{\rm(a)}]
\item
$(0:_Q\m_Q^{s-1}) =  \m_Q^{2}+(0:_Q \m_Q^{k+1})$
\smallskip
\item $\lambda\left(\dfrac{(0:_Q\m_Q^{s-1})}{\m_Q^{2}}\right) = \edim(B).$
\end{enumerate}
\end{lemma}
\begin{proof}\hfill{}\\
a) By Remark \ref{FPrmk}(a), it is enough to prove that $\m_Q^{s-1} =  (0:_Q\m_Q^{2})\cap\m_Q^{k+1}$.
Since $s -1 \geq k+1$, we have $\m_Q^{s-1} \subset  (0:_Q\m_Q^{2})\cap \m_Q^{k+1}$. 

In order to prove the other containment, let $0 \neq z \in \m_Q^{k+1}$ be such that $z \m_Q^{2} = 0$. Then
$z^*\in (0:_G \m_G^{2}) \cap \m_G^{k+1} = (0:_G (\m_A^{2} + \m_B^{2})) \cap \m_G^{k+1}  = (0:_G \m_A^{2}) \cap (0:_G \m_B^{2}) \cap \m_G^{k+1}$.
Since $\m_G^{k+1} \subset (0:_G \m_B^{2})$, we get $z^* \in (0:_G \m_A^{2}) \cap \m_G^{k+1} = ((0:_A \m_A^{2}) + \m_B) \cap \m_A^{k+1}$.
Thus, there exist homogeneous elements $\alpha \in (0:_A \m_A^{2})$, and $\beta \in \m_B$, such that $z^* = \alpha + \beta$. Since $\m_B^{k+1} = 0$, and $\deg(z^*) \geq k+1$, we see that $\beta = 0$.

This shows that $z^* \in (0:_A \m_A^{2}) = \m_A^{s-1} = \m_G^{s-1}$, by Remark \ref{FPrmk}. Thus, $z \in \m_Q^{s-1}$, which proves (a).

\noindent
b) Remark \ref{FPrmk} implies that $$\lambda\left(\dfrac{(0:_Q\m_Q^{s-1})}{\m_Q^{2}}\right) =\lambda(Q/\m_Q^{s-1})-\lambda(\m_Q^{2})
=\lambda(Q/\m_Q^{2})-\lambda(\m_Q^{s-1})
=\lambda(Q/\m_Q^{2})-\lambda(\m_A^{s-1}),$$
where the last equality follows since $s -1  \geq k+1$.
Now, $\lambda(\m_A^{s-1}) = \lambda(0:_A \m_A^{2}) = \lambda(A/\m_A^{2})$, and hence $$\lambda\left((0:_Q\m_Q^{s-1})/\m_Q^{2}\right) = \lambda\left(Q/\m_Q^{2}\right) - \lambda\left(A/\m_A^{2}\right) = \lambda\left(B/\m_B^{2}\right) - 1,$$ proving (b), since $\edim(B)+1=\lambda\left(B/\m_B^{2}\right)$.
\end{proof}

\begin{theorem}\label{EXJ}
With notation as in Setup \ref{setup}, let $J = \langle z_1 \ldots, z_n \rangle \subset (0:_Q \m_Q^{k+1})$ be such that the images of $z_1, \ldots, z_n$ form a $\sk$-basis of $((0:_Q\m_Q^{k+1})+\m_Q^2)/\m_Q^2$, and $I = \langle y_1,\ldots, y_m\rangle$ be such that $\m_A$ is minimally generated by $\{ y_1^*,\ldots, y_m^*\}$. Then the following hold:
\begin{enumerate}[{\rm(a)}]
\item $\edim(B) = \mu(J) = n$, $J + \m_Q^2 = 0:_Q \m_Q^{s-1}$, and if $k =1$, then $0:_Q \m_Q^{s-1} \subset J + (0:_Q J)$. 
\item $\m_B$ and $\m_Q$ are minimally generated by $\{z_1^*,\ldots,z_n^*\}$ and $\{y_1,\ldots,y_m,z_1,\ldots,z_n\}$ respectively. In particular, $\m_Q = I + J$.
\item $J\m_Q^{k}=\soc(Q)$, and $J^k \neq 0$.
\item $IJ \subset \m_Q^3$, $I^kJ + J^{k+1} \subset \soc(Q)$, and $I^aJ^b=0$ for $a+b=k+1$, and $2 \leq b \leq k$.
\item $\m_Q^i=I^i$ for $i\geq k+1$.
\end{enumerate}
\end{theorem}

\begin{proof}\hfill{}\\
a) By the choice of the $z_i$'s, the images of $z_1,\ldots,z_n$ are linearly independent in $\m_Q/\m_Q^2$, and hence in $J/\m_Q J$. Thus $\mu(J) = n$. Moreover, by Lemma \ref{stabl}(a), we get $J + \m_Q^2 = (0:_Q\m_Q^{s-1})$, and hence Lemma \ref{stabl}(b) gives $\mu(J)=\edim(B)$. 

Observe that if $k = 1$, then $0: (J + \m_Q^{k+1}) = 0: (J + \m_Q^{2}) = 0:_Q(0:_Q \m_Q^{s-1}) = \m_Q^{s-1}$. Thus, we have $(0:_Q J) \cap J \subset (0:_Q J) \cap (0:_Q \m_Q^{k+1}) \subset 0:_Q(J + \m_Q^{k+1}) =  \m_Q^{s-1}$. Since $Q$ is Gorenstein, this shows that $0:_Q \m_Q^{s-1} \subset J + (0:_Q J)$, proving (a).

\noindent
b) Note that if we prove $\m_B$ is minimally generated by $\{z_1^*,\ldots,z_n^*\}$, then $\m_G$ is minimally generated by $\{y_1^*,\ldots,y_m^*, z_1^*,\ldots,z_n^*\}$, and hence $\m_Q=\langle \y,\z\rangle$. Thus, in order to prove (b), since $\edim(B) = n$, it is enough to show that $\{z_1^*+\m_B^2,\ldots,z_n^*+\m_B^2\}$ is a linearly independent set in $\m_B/\m_B^2$. Firstly, note that $z_j^*\in (0:_G \m_G^{k+1})= \m_A^{s-k} + \m_B$ by Remark \ref{FPrmk}(c). Since $\deg(z_j^*)=1$ and $s - k > 1$, we get $z_j^*\in \m_B$. 

Now, let $\alpha = \sum_{i=1}^n \alpha_i z_i^*$, where $\alpha_i \in \sk$. Then $\alpha = 0$ or $\deg(\alpha) = 1$. Thus $\alpha \in \m_B^2$ implies $\alpha = 0$ in $\m_B$, and hence $\m_G$. If $a_i \in Q$ is a lift of $\alpha_i$, this implies that $\sum_{i=1}^n a_i z_i \in \m_Q^2$. In particular, $a_i \in \m_Q$, by the independence of the images of the $z_i$'s modulo $\m_Q^2$. This forces $\alpha_i = 0$ for each $i$, and hence $\{z_1^*+\m_B^2,\ldots,z_n^*+\m_B^2\}$ is a linearly independent set in $\m_B/\m_B^2$. This proves (b). 

\noindent
c) We first show that $J^k$ is not contained in $\m_Q^{k+1}$. If not, for non-negative integers $i_1, \ldots, i_n$ such that $i_1 + \cdots + i_n = k$, we have $(z_1^*)^{i_1}\cdots(z_n^*)^{i_n} = 0$, i.e, $\m_B^k = 0$, which is a contradiction to $\ell\ell(B) = k$. In particular, since $J^k\m_Q^2=0$, it follows from Lemma \ref{stabl} that $J^k$ is not contained in $\m_Q^{s-1}=\m_Q^{k+1}\cap(0:_Q \m_Q^2)$. This shows that $J^k\neq 0$. 

Furthermore, since $J^k\subset \m_Q^k\cap (0:_Q \m_Q^2)$, we see that $\m_Q^{s-1}\subsetneq  \m_Q^k\cap (0:_Q \m_Q^2)$.
Hence, by Remark \ref{FPrmk}(a), we have $(0:_Q \m_Q^k)+ \m_Q^2 \subsetneq (0:_Q \m_Q^{s-1})=J+\m_Q^2$. In particular, $J\m_Q^{k} \neq 0$, and hence, $\soc(Q) \subset J\m_Q^k$. Since $J \m_Q^{k+1} = 0$, we get the other containment, proving (c).

\noindent
d) Now, $G \simeq A \times_\sk B$ implies $y_i^*z_j^*=0$ in $G$, which forces $y_iz_j\in \m_Q^3$ for all $i$ and $j$. Hence $IJ \subset \m_Q^3$.
Moreover, $J\m_Q^k = \soc(Q)$ implies that $I^kJ + J^{k+1} \subset \soc(Q)$. 
Finally, since $a + b = k+1$, and $2 \leq b \leq k$, we have $I^aJ^b = J(IJ)(I^{a-1}J^{b-2}) \subset J\m_Q^3\m_Q^{k-2} = J \m_Q^{k+1} = 0$, proving (d).

\noindent
e) It is enough to show that $\m_Q^{k+1}=I^{k+1}$ as $\m_Q = I + J$, and $J\m_Q^{k+1} = 0$. Now, by (d), we have $\m_Q^{k+1} = I^{k + 1} + \soc(Q)$. We now claim that $I^{k+1} \neq 0$. This is true since if $I^{k+1}=0$, then $\m_Q^{k+2}=0$, which contradicts the fact that $s > k+1$. Since $Q$ is Gorenstein Artin, and $I^{k+1} \neq 0$, we have $\soc(Q)\subset I^{k+1}$, and hence $\m_Q^{k+1} = I^{k+1}$.
\end{proof}

\section{Associated Graded Rings and Decomposability}\label{Gr}

As seen in Example \ref{FPex1}, given a Gorenstein Artin local ring $Q$, if $\gr(Q) \simeq A \times_\sk B$ for graded rings $A$ and $B$, then $Q$ need not be decomposable as a connected sum over $\sk$. However, if $A$ is a graded Gorenstein quotient of $\gr(Q)$ with $\ell\ell(A) > \ell\ell(B)+1$, and one of the following holds: (i) $\ell\ell(B) = 1$ or (ii) $\edim(A) = 1$ and $\ell\ell(B)=2$, then the converse of Proposition~\ref{Prop2} holds, as can be seen in Theorem \ref{StructureTheorem} and Corollary \ref{Loewy(B)=2}.

Before we head towards the proofs of these theorems, we show that such results do not hold, without further hypothesis, for $\ell\ell(B) \geq 3$.

\begin{example}{\rm
Let $Q=\sk[X,Y]/\langle X^3+Y^3+XY,Y^4\rangle$. Then we get $\gr(Q)\simeq A \times_{\sk} B$, where $A =\sk[X]/\langle X^9\rangle$ and $B = \sk[Y]/\langle Y^4\rangle$. Note that $A$ is graded Gorenstein, $\edim(A) = 1$, $\ell\ell(A) = 8 > \ell\ell(B) + 1$ since $\ell\ell(B) = 3$.

In this example, $H_Q(2)=3\geq \binom{2}{2}+2$, so by Remark~\ref{indecomp.rmk}(b), $Q$ is indecomposable as a connected sum over $\sk$.}
\end{example}

\subsection{The \boldmath$\ell\ell(B) = 1$ case}

We begin with the following proposition. Note that the statement is true more generally, in particular, for any standard graded $\sk$-algebra $G$.

\begin{proposition}\label{GLS}
Let $(Q,\m_Q,\sk)$ be a Gorenstein Artin local ring with $\ell\ell(Q) \geq 2$, and $G = \gr(Q)$. Then the following are equivalent:
\begin{enumerate}[\rm i)]
\item $G$ is as in Setup \ref{setup}, with $\ell\ell(B) = 1$. 
\item There exists a graded Gorenstein ring $A$ and a surjective ring homomorphism $\pi\colon  G \rightarrow A$
such that $\ker(\pi) \cap (\m_G)^2 =0$.
\item $\dim_{\sk} (\soc(G)\cap (\m_G)^2)=1$.
\end{enumerate}
In particular, if {\rm (i)} holds, then we have $A \simeq G/\langle \soc(G)\cap G_{1}\rangle$, $\m_B = \langle \soc(G) \cap G_1\rangle$, \\
$\lambda(G) - \lambda(A) = \edim(G) - \edim(A) = \type(G) - 1 = \edim(B)$, and $\ell\ell(A) = \ell\ell(G)$.
\end{proposition}

\begin{proof}
(i) $\Rightarrow$ (ii): By Remark \ref{FP}, $A \simeq G/\m_B$. Since $\m_B^2 = 0$, (ii) holds.
\vskip 3pt

\noindent
(ii) $\Rightarrow$ (iii): By (ii), we see that $\pi |_{(\m_G)^i}: (\m_G)^i \longrightarrow (\m_A)^i$ is an isomorphism for each
$i \geq 2$. In particular, if $\soc(A) = (\m_A)^s$, then $s = \ell\ell(A) = \ell\ell(G) \geq 2$ and $\dim_{\sk}((\m_G)^s) = 1$.

Suppose $z \in \soc(G) \cap (\m_G)^2$. Then $\pi(z) \in \soc(A) = (\m_A)^s$. Since $s \geq 2$, this forces $z \in (\m_G)^s$.
Thus $(\m_G)^s \subset \soc(G) \cap (\m_G)^2 \subset (\m_G)^s$, proving (iii).
\vskip 3pt

\noindent
(iii) $\Rightarrow$ (i): Since $G \simeq A \times_\sk B$, where $A \not\simeq \sk \not\simeq B$, we see that $G$ is not Gorenstein. Hence (iii) implies that there is a 
$\sk$-basis, say $\{\overline{z}_1,\ldots,\overline{z}_n\}$, for $(\soc(G) + (\m_G)^2)/(\m_G)^2$, where
$n = \type(G) - 1 \geq 1$. Extend this to a $\sk$-basis
$\{\overline{y}_1, \ldots, \overline{y}_m, \overline{z}_1,\ldots,\overline{z}_n\}$ of
$\m_G/(\m_G)^2$, and lift it to a minimal generating set $\{\y, \z\}$ of $\m_G$ in $G_1$.
Since $\langle \y\rangle \cap \langle \z\rangle = 0$ and $\langle \y\rangle + \langle \z\rangle = \m_G$, 
Remark \ref{FP} implies that  $G \simeq A \times_\sk B$,
where $A = G/\langle \z\rangle$ and $B = G/\langle \y\rangle$ are graded $\sk$-algebras.

Since $z_i \in \soc(G)$ for each $i$, their images in $B$, which are the generators of $\m_B$, are in $\soc(B)$, and in
particular, $\m_B = \soc(B)$. Thus $\m_B^2 = 0$, and $\type(B) = n$. Therefore, by Remark \ref{FP},  $\type(A) = 1$, i.e., $A$ is a graded Gorenstein $\sk$-algebra.
\vskip 3pt

The last part follows from the proof above, Remark \ref{FP} and the facts that $\type(A) = 1$ and
$\edim(B) = \type(B) = \lambda(B) - 1$.
\end{proof}

\begin{remark}{\rm
The proof of (iii) $\Rightarrow$ (i) in the above proposition follows from \cite[Lemma 1.6]{AAM}. We reprove it here for the sake of completeness. 
}\end{remark}

We now prove the following proposition, which is crucial in our proof of Theorem \ref{StructureTheorem}.

\begin{proposition} \label{basisprop} With notation as in Setup \ref{setup}, let $k = 1$, and $J = \langle z_1,\ldots, z_n \rangle$ be as in Theorem \ref{EXJ}. Then $J + (0:_Q J) = \m_Q$, and $\mu(0:_Q J) = \edim(Q) - n$.
\end{proposition}
\begin{proof}
Let $I = (0:_Q J)$. Since $Q$ is Gorenstein, in order to prove $J + I = \m_Q$, it is enough to show that $I \cap J = \soc(Q)$, and since $0 \neq J \neq Q$, we only need to prove $I \cap J \subset \soc(Q)$. 

Since $s \geq 3$, by Theorem \ref{EXJ}(a), it follows that $I \cap J \subset  \m_Q^{s-1} \subset \m_Q^2$. Therefore, the linear independence of $z_1,\ldots,z_n$ modulo $\m_Q^2$ shows that if $\sum_{i=1}^n a_iz_i \in I \cap J$, then $a_i \in \m_Q$. Thus,
$I \cap J \subset J \m_Q = \soc(Q)$, by Theorem \ref{EXJ}(c), proving $J + I = \m_Q$. 

Now $J\m_Q= \soc(Q)$, and $\mu(J) = n$ imply that $\lambda(J)=\lambda(J/J\m_Q) + 1= n+1$. Furthermore, $I + J = \m_Q$, and $I J = 0$ show that $\m_Q^2 = I \m_Q + \soc(Q)$. Since $s \geq 3$, we have $I\m_Q \neq 0$, and hence $\m_Q^2 = I \m_Q$. Thus, $\lambda(0:_QJ) = \lambda(Q/J)$ gives
\begin{align*}
\mu(I) = \lambda(I/\m_Q^2)& =\lambda(I)-\lambda(\m_Q^2)=\lambda(Q/J)-\lambda(\m_Q^2) =\lambda(Q/\m_Q^2)-\lambda(J)\\
& = (1+\edim(Q)) - (n + 1) = \edim(Q) - n
\end{align*}
proving the proposition.
\end{proof}

\begin{remark}\label{Imp} {\rm
With notation as above, we see that
$\langle \soc(G) \cap G_1\rangle =\m_B= \langle z_1^*, \ldots, z_n^*\rangle$ by Theorem \ref{EXJ}(b) and Proposition \ref{GLS}.
}\end{remark}

We are now ready to state and prove one of the main theorems of this section.

\begin{theorem}\label{StructureTheorem}
Let $(Q, \m_Q,\sk)$ be a Gorenstein Artin local ring, $G = \gr(Q)$ and $k$ be as in Setup \ref{setup}. If $k = 1$, then $Q$ decomposes as a connected sum over $\sk$. Moreover, we can write $Q \simeq R \#_\sk S$, where $(R,\m_R,\sk)$ and $(S,\m_S,\sk)$ are Gorenstein Artin local rings such that:
\begin{enumerate}[{\rm a)}]
\item $\edim(S) = \type(G) - 1$ and $\ell\ell(S) = 2$.
\item $\gr(R)\simeq G/\langle \soc(G)\cap G_{1}\rangle$. In particular, $\ell\ell(R) = s$ and $H_R(i) = H_Q(i)$ for $2 \leq i \leq s$.
\item If $m = \edim(R)$ and $n = \edim(S)$, then $[\p^R(t)]^{-1}=[\p^Q(t)]^{-1}-[\p^S(t)]^{-1} + 1 - \phi_{m,n}t^2$,
where $ \phi_{m,n}$ is given as in Theorem \ref{Characterization}(d) and, $\p^S(t)$ is as in Remark \ref{PS}(c).\\
Thus, $\p^Q(t)$ is rational in $t$ if and only if $\p^R(t)$ is so.
\end{enumerate}
\end{theorem}

\begin{proof}
By Proposition~\ref{basisprop}, we
have $\m_Q = \langle y_1,\ldots,y_m, z_1,\ldots,z_n\rangle$ where $n = \type(G) - 1 \geq 1$,
$\y\cdot \z=0$ and $\langle \z  \rangle^2 \neq 0 = \langle \z \rangle^3$. Let $\widetilde{Q}/I_Q$ be a Cohen presentation of $Q$, where $\m_{\widetilde Q} = \langle \Y,\Z\rangle$, $\langle \Y \cdot \Z\rangle + \langle \Z  \rangle^3 \subset I_Q \subset\m_{\widetilde Q}^2$, and $\langle \Z  \rangle^2 \not\subset I_Q$. Observe that $\langle \z \rangle^3 = 0$ and $s \geq 3$ force $m \geq 1$.

Now, if $J_R = (I_Q \cap \langle \Y\rangle) + \langle \Z \rangle$ and $J_S = (I_Q \cap \langle \Z\rangle) + \langle \Y \rangle$, then by Theorem \ref{Characterization}, $R = \widetilde Q/J_R$ and
$S= \widetilde Q/J_S$ are Gorenstein Artin such that $Q\simeq R\#_\sk S$. Moreover, $J_S \cap \langle \Z \rangle = I_Q \cap \langle \Z \rangle$. Hence, $\langle \Z  \rangle^3 \subset J_S$ and $\langle \Z  \rangle^2 \not\subset J_S$. Thus $\m_S^3 = 0 \neq \m_S^2$, i.e., $\ell\ell(S) = 2$, proving (a).

Since $\ell\ell(Q) \geq 3 > \ell\ell(S)$, we see that $\ell\ell(R) = \ell\ell(Q) \neq \ell\ell(S)$. Hence, by (the proof of) Proposition \ref{Prop2}, $G \simeq \gr(R) \times_\sk \gr(S/\soc(S))$. In particular, $\gr(R) \simeq G/\langle z_1^*,\ldots, z_n^*\rangle$. By Remark \ref{Imp}, we get $\gr(R) \simeq G/\langle \soc(G)\cap G_{1}\rangle$, proving (b).

The remaining statements follow from Theorem \ref{Characterization}.
\end{proof}

\begin{remark}{\rm
The condition that $A$ is a graded Gorenstein ring is necessary in the above theorem. Example \ref{FPex1} gives a counter-example when $A$ is not graded Gorenstein.
}\end{remark}

The next corollary, which gives a sufficient condition for $\gr(Q)$ to be Gorenstein, is an immediate consequence of Proposition \ref{GLS} and the above theorem.

\begin{corollary}
Let $(Q, \m_Q,\sk)$ be a Gorenstein Artin local ring which is indecomposable as a connected sum, and let $G = \gr(Q)$. If $\ell\ell(Q) \geq 3$, and $\dim_{\sk} (\soc(G)\cap (\m_G)^2)=1$, then $G$ is Gorenstein.
\end{corollary}

\subsection{The \boldmath$\edim(A) = 1$ case}
We now focus on the case where $\edim(A) = 1$ in Setup \ref{setup}. We identify some conditions on the ideals $I$ and $J$ defined in Theorem \ref{EXJ}, which force  $Q$ to be a connected sum over $\sk$. As an immediate consequence, we see that $Q$ is a connected sum over $\sk$, where $\edim(A)=1$ and $\ell\ell(B)=2$, assuming $\ell\ell(A) \geq 4$.

\begin{proposition}\label{LL=k}
With the notation as in Setup \ref{setup}, let $\edim(A) = 1$, and let $I = \langle y \rangle$ and $J = \langle z_1, \ldots, z_n \rangle$ be the ideals defined in Theorem \ref{EXJ}. If $IJ\subset \m_Q^{k+1}$, then there is an ideal $J'$ in $Q$ such that $IJ' = 0$, $I + J' = \m_Q$, $\mu(J') = \edim(Q) - 1$, and $(J')^{k+1} = \soc(Q)$.
\end{proposition}

\begin{proof}
By Theorem \ref{EXJ}(c), we have $I^kJ\subset \soc(Q) = \langle y^s \rangle$. Thus for each $j = 1, \ldots, n$, there exists $u_j \in Q$, such that $y^k z_j = u_j y^s$. Define $z_j' = z_j -u_j y^{s-k}$ and set $J_1=\langle z'_1,\ldots,z'_n\rangle$. 

Observe that $J \m_Q^{k+1} = 0$ implies $J_1 \m_Q^{k+1} = 0$. Moreover, $J+\m_Q^2=J_1+\m_Q^2$ implies, by the choice of $J$, that the images of $z_1',\ldots,z_n'$ form a $\sk$-basis for $((0:_Q\m_Q^{k+1})+\m_Q^2)/\m_Q^2$. Thus $J_1$ satisfies the same hypothesis, and hence the same conclusions, of Theorem \ref{EXJ} as $J$. Furthermore, we also have $I^kJ_1 = 0$.

We claim that this forces $I^{k-1}J_1 \subset \soc(Q)$. In order to see this, since $\m_Q = I + J_1$, and $I^kJ_1 = 0$, it is enough to show that $I^{k-1}J_1^2 = 0$. Now, $I^{k-1}J_1^2 \subset I^{k-1}(J+I^{s-k})J_1 \subset \m_Q^{k+1}J_1$ by the given hypothesis, since $s > k + 1$, and $k \geq 2$. But, by Theorem \ref{EXJ}(c) applied to $J_1$, we have $J_1 \m_Q^{k+1} = 0$, proving $I^{k-1}J_1 \subset \soc(Q)$. 

The same argument shows that there exists an ideal $J_2$, satisfying the same hypothesis, and conclusions, as $J$ in Theorem \ref{EXJ}, such that $I^{k-1}J_2 = 0$. Repeating this process $k$ times, we get ideals $J_1$, \ldots, $J_k$, each satisfying the same conclusions of Theorem \ref{EXJ} as $J$, and further satisfying $I^{k-i+1}J_i = 0$ for $1 \leq i \leq k$. 

Thus, we have $IJ_k = 0$. Moreover, Theorem \ref{EXJ} applied to $J' = J_k$ gives $I + J' = \m_Q$, and $\mu(J') = n = \edim(Q) - 1 \geq 1$. Finally, $\soc(Q) = J'\m_Q^k = J'(I + J')^k = (J')^{k+1}$, proving the result.
\end{proof}

The following important theorem is a consequence of Proposition \ref{LL=k} and Theorem \ref{Characterization},
\begin{theorem}\label{Loewy(B)=k}
With the notation as in Setup \ref{setup}, let $\edim(A) = 1$, and let $I$ and $J$ be the ideals defined in Theorem \ref{EXJ}. If $IJ\subset \m_Q^{k+1}$, then there are Gorenstein Artin local rings $(R,\m_R,\sk)$ and $(S,\m_S,\sk)$ such that $Q\simeq R\#_{\sk}S$ where
\begin{enumerate}[{\rm(a)}]

\item $R$ is a hypersurface, with $\ell\ell(R) = s$, $\gr(R) \simeq A$, and $H_R(i) = H_Q(i)=1$ for $k+1 \leq i \leq s$,
\item $\ell\ell(S)=k+1$, and $\displaystyle{{\large \frac{1}{\p^{Q}(t)} =  \frac{1}{\p^S(t)}}- t}$.\\
Thus, $\p^Q(t)$ is rational in $t$ if and only if $\p^S(t)$ is so.
\end{enumerate}

\end{theorem}
\begin{proof}
By Proposition \ref{LL=k}, letting $I = \langle y \rangle$, there exists an ideal $J' = \langle z_1,\ldots, z_n \rangle$ with $n = \edim(Q) - 1 \geq 1$, such that $\m_Q =\langle y,z_1,\ldots,z_n\rangle$, where $y\cdot{\bf{z}}=0$ and $\langle{\bf{z}}\rangle^{k+1}\neq 0=\langle{\bf{z}}\rangle^{k+2}$. 
Let $\widetilde{Q}/I_Q$ be a Cohen presentation of $Q$, where $\m_{\widetilde Q} = \langle Y,\Z\rangle$, where $Y$ and $\Z$ are lifts of $y$ and $\z$ respectively. Then $\langle Y \cdot \Z\rangle + \langle \Z  \rangle^{k+2} \subset I_Q \subset\m_{\widetilde Q}^2$, and $\langle \Z  \rangle^{k+1} \not\subset I_Q$.

Now, if $J_R = (I_Q \cap \langle Y\rangle) + \langle \Z \rangle$ and $J_S = (I_Q \cap \langle \Z\rangle) + \langle Y \rangle$, then by Theorem \ref{Characterization}, $R = \widetilde Q/J_R$ and
$S= \widetilde Q/J_S$ are Gorenstein Artin such that $Q\simeq R\#_\sk S$. Moreover, $J_S \cap \langle \Z \rangle = I_Q \cap \langle \Z \rangle$. Hence, $\langle \Z  \rangle^{k+2} \subset J_S$ and $\langle \Z  \rangle^{k+1} \not\subset J_S$. Thus $\m_S^{k+2} = 0 \neq \m_S^{k+1}$, i.e., $\ell\ell(S) = k+1$.

Note that $\edim(R)=1$, and $R$ is Gorenstein Artin, hence $R$ is a hypersurface. Since $\ell\ell(Q) = s > k+1 = \ell\ell(S)$, we see that $\ell\ell(R) = s \neq \ell\ell(S)$. Hence, by Theorem \ref{Characterization}(e), we get $G \simeq \gr(R) \times_\sk \gr(S/\soc(S))$. In particular, $\gr(R) \simeq G/\langle z_1^*,\ldots, z_n^*\rangle \simeq A$, since $\m_B=\langle z_1^*,\ldots, z_n^*\rangle$ by Theorem \ref{EXJ}.

The other conclusions follow from Theorem \ref{Characterization}(b) and (d).
\end{proof}

By Theorem \ref{EXJ}(d), $IJ \subset \m_Q^3$. Hence, with $k = 2$ in the above theorem, the following corollary is immediate.

\begin{corollary}\label{Loewy(B)=2}
Let $(Q,\m_Q,\sk)$ be a Gorenstein Artin ring with $\ell\ell(Q)=s\geq 4$. Suppose $G=\gr(Q)\simeq A\times_{\sk}B$ with $\edim(A)=1$ and $\ell\ell(B)=2$. Then there are Gorenstein Artin local rings $(R,\m_R,\sk)$ and $(S,\m_S,\sk)$ such that $Q\simeq R\#_{\sk}S$, where $R$ is a hypersurface and $S$ is a short Gorenstein ring. Moreover, $\p^Q(t)$ is rational in $t$ if and only if $\p^S(t)$ is so.
\end{corollary}

With notations as in Setup \ref{setup}, the conditions $IJ \subset \m^{k+1}$, and $\edim(A) = 1$ are both necessary in Theorem \ref{Loewy(B)=k}, as can be seen below. In both the following examples, we denote $x,y,z$ as images of $X,Y,Z$ respectively in $Q$. Note that $Q$ is indecomposable as a connected sum over $\mathbb{Q}$ by Remark~\ref{indecomp.rmk}(a), since it is complete intersection. 

\begin{example}\label{4.16}\hfill{}\\{\rm
(a) In this example, we have $s = 5$, $k = 3$ and $\edim(A) = 1$, but $IJ\not\subset \m_Q^{4}$. 

Let $Q=\mathbb{Q}[X,Y,Z]/\langle X^{4}-YZ,Y^3-XZ,Z^3-XY\rangle$. Then $\gr(Q)\simeq A \times_{\sk} B$, where $A \simeq \mathbb{Q}[X]/\langle X^{6}\rangle$ and $B \simeq \mathbb{Q}[Y,Z]/\langle Y^{4},YZ,Z^{4}\rangle$. Furthermore, we have $I=\langle x\rangle$, $J=\langle y,z\rangle$, and hence $IJ\not\subset \langle x,y,z\rangle^4$.

\noindent
(b) In this example, we have $s = 5$, $k = 3$, $IJ\subset \m_Q^{4}$, and $\edim(A)=2$.

Let $Q=\mathbb{Q}[X,Y,Z]/\langle X^4-YZ,Y^4-XZ,Z^3-X^2 \rangle$. Then $\gr(Q)=A \times_\sk B$, where $A = \mathbb{Q}[X,Y]/\langle X^2,Y^{5}\rangle$, and $B = \mathbb{Q}[Z]/\langle Z^4\rangle$. In this case, we have $I=\langle x,y\rangle$, $J=\langle z\rangle$ and $IJ\subset \langle x,y,z\rangle^4$.
}
\end{example}

\section{Some Applications}\label{SS}

\subsection{Short and Stretched Gorenstein Rings}

In her paper on stretched Gorenstein rings, Sally proved a structure
theorem (\cite[Corollary 1.2]{Sa}) for a stretched Gorenstein
local ring $(Q,\m_Q,\sk)$ when char$(\sk) \neq 2$. The description of the defining ideal of $Q$ shows that $Q$ can be decomposed as a connected sum over $\sk$.

Elias and Rossi proved a similar structure theorem (\cite[Theorem 4.1]{ER})
for a short Gorenstein local $\sk$-algebra $(Q,\m_Q,\sk)$ when $\sk$ is algebraically closed and char$(\sk) = 0$, which shows that $Q$ decomposes as a connected sum over $\sk$.

Theorem \ref{StructureTheorem} generalizes these two results, which can be seen as follows:

\begin{proposition}\label{maincor}
Let $(Q,\m_Q,\sk)$ be either a short or a stretched Gorenstein Artin ring and $G = \gr(Q)$. Set $\ell\ell(Q) = s$. If $s \geq 3$, then $\edim(Q) = H_Q(s-1) + \type(G) - 1$, and $\dim_{\sk}(\soc(G)\cap \m_G^2)=1$. In particular, $G$ is as in Setup \ref{setup}, with $\ell\ell(B) = 1$.
\end{proposition}

\begin{proof}
Let $Q_0 = G/C$ be the quotient of $G$ as defined by Iarrobino
(see Remark \ref{Iarrobino}). Note that $C^2=0$. Since $H_G(i) = H_{Q_0}(i)$ for $i = s-1, s$ by Remark
\ref{Iarrobino}, the fact that the 
Hilbert function of a graded Gorenstein $\sk$-algebra is palindromic gives us the following:\\
i) Let $Q$ be a short Gorenstein ring with $H_Q = (1,h,n,1)$. Then $H_{Q_0} = (1,n,n,1)$.\\
ii) If $Q$ is stretched with $H_Q =
(1,h,1,\ldots,1)$, then $H_{Q_0} = (1,1,1,\ldots,1)$.

Thus if we take $A$ to be $Q_0$ in Proposition \ref{GLS}(iii), we see that
$\dim_{\sk}(\soc(G)\cap \m_G^2)=1$, and the formula for $\edim(Q)$ holds since $\edim(A) =  H_Q(s-1)$. Finally, by Proposition \ref{GLS}, $G$ is as in Setup \ref{setup} with $\ell\ell(B) = 1$.
\end{proof}

\begin{theorem}\label{stretched}
Let $(Q, \m_Q,\sk)$ be Gorenstein Artin with $G = \gr(Q)$ and $\ell\ell(Q) \geq 3$. Then $Q$ is stretched if and only if
$\dim_{\sk}(\soc(G)\cap \m_G^2)=1$ and $\type(G) = \edim(Q)$. In particular, $Q$ decomposes as a connected sum over $\sk$, and
$[\p^Q(t)]^{-1} = 1 - \edim(Q) t + t^2$ when $\edim(Q) \geq 2$.
\end{theorem}
\begin{proof}
Note that if $Q$ is stretched, the required properties of $G$ hold by Proposition \ref{maincor}. For the converse,
assume that $\dim_{\sk}(\soc(G)\cap \m_G^2)=1$ with $\edim(Q) = \type(G)$. Then by Proposition \ref{GLS}, we have
$G \simeq A \times_\sk B$ with $\edim(B)=\edim(Q)-1$ and $\edim(A)=1$. Since $\ell\ell(Q) \geq 3$, Theorem \ref{StructureTheorem}(b)
and (c) show that $Q$ is stretched, and give the formula for $\p^Q(t)$.
\end{proof}

\begin{remark}{\rm
It is shown in \cite{CENR} if $(Q, \m_Q,\sk)$ is a short Gorenstein Artin $\sk$-algebra with
Hilbert function $H_Q = (1,h,n,1)$, then $\p^Q(t)$ is rational when $n \leq 4$, with the assumption that $\sk$ is an algebraically closed field of
characteristic zero. In Theorem \ref{n=4}, we show the same is true for all Gorenstein Artin local rings with $n \leq 4$.
}\end{remark}

\begin{theorem}\label{n=4}
Let $(Q, \m_Q,\sk)$ be a short Gorenstein Artin local ring with
Hilbert function $H_Q = (1,h,n,1)$. Then $Q$ is a connected sum.
Furthermore, if  $n\leq 4$, then $\p^Q(t)$ is rational.
\end{theorem}
\begin{proof}
By Proposition \ref{maincor} and Theorem \ref{StructureTheorem}, there exist Gorenstein Artin local rings $R$ and $S$
such that $Q \simeq R \#_\sk S$, where $\ell\ell(S) \leq 2$ and $\edim(R) = n \leq 4$. Thus, Remark \ref{PS}(a) shows that $\p^R(t)$ is
rational, and therefore, by Theorem \ref{StructureTheorem}(c), $\p^Q(t)$ is rational.
\end{proof}

\subsection{Rationality of Poincar\'e Series}

In this subsection, we see some conditions which force the rationality of $\p^Q(t)$. We begin with the following consequence of Theorem \ref{StructureTheorem}.

\begin{proposition}\label{Corr}
Let $(Q, \m_Q,\sk)$ be a Gorenstein Artin local ring. If $G=\gr(Q)$ and $k$ are as in Setup \ref{setup} with $k = 1$, then
$\p^Q(t)$ is a rational function of $t$ in the following situations:
\begin{enumerate}[{\rm i)}]
\item  $\edim(Q) - \type(G) \leq 3$.
\item $\lambda(Q) - \type(G) \leq 10$.
\end{enumerate}
\end{proposition}

\begin{proof}
By Theorem \ref{StructureTheorem}(a), $Q \simeq R \#_\sk S$ for Gorenstein Artin local rings $(R,\m_R,\sk)$ and $(S,\m_S,\sk)$
with $\ell\ell(S) \leq 2$ and $\gr(R) \simeq A = G/\langle \soc(G)\cap G_{1}\rangle$.

Now, by Proposition \ref{GLS}, $A$ is a graded Gorenstein Artin
$\sk$-algebra such that $\ell\ell(A) = \ell\ell(Q)$, $\edim(A) = \edim(Q) - \type(G) + 1$ and $\lambda(A) = \lambda(Q) - \type(G) + 1$.

(ii) The assumption that $\lambda(Q) - \type(G) \leq 10$ implies that $\lambda(A) \leq 11$. Since $A$ is graded Gorenstein, and
hence has a palindromic Hilbert function, the hypothesis that $\ell\ell(Q) \geq 3$ forces $\edim(A) \leq 4$. Thus (ii) reduces
to (i).

(i) In this case, $\edim(R) = \edim(A) \leq 4$. Hence, by Remark \ref{PS}(a), $\p^R(t)$ is
rational. Thus, by Theorem \ref{StructureTheorem}(c), $\p^Q(t)$ is rational.
\end{proof}

\begin{corollary}\label{poincare} Let $(Q,\m_Q,\sk)$ be a Gorenstein Artin local ring with $\ell\ell(Q)\geq 4$. Suppose $\gr(Q)\simeq A\times_{\sk}B$ with $\edim(A)=1$ and $\ell\ell(A)>\ell\ell(B)=2$.  If $H_Q(2)\leq 5$, then $\p^Q(t)$ is rational.
\end{corollary}

\begin{proof}
By Corollary \ref{Loewy(B)=2} we have $\ell\ell(S)=3<s=\ell\ell(R)$. It then follows from Theorem \ref{Characterization}(b) that $H_Q(2) = H_R(2) + H_S(2)$. Since $H_R(2)=1$, we get $H_S(2)\leq 4$. Now applying Theorem \ref{n=4} for $S$, we get
$\p^S(t)$ is rational. Thus, $\p^Q(t)$ is rational by Corollary \ref{Loewy(B)=2}.
\end{proof}

We end this article by exhibiting another class of rings which are connected sums. This is a natural extension of the case of stretched rings.

\begin{proposition} 
Let $(Q,\m_Q,\sk)$ be Gorenstein Artin, and $G= \gr(Q)$ be as in Setup \ref{setup}. Further, assume that $\ell\ell(Q) \geq 4$, and $\mu(\m_Q^3)=1$, i.e., $H_Q$ is of the form $(1,h,n,1,1,\ldots,1)$. Then $Q$ can be decomposed as a connected sum over $\sk$. In particular, if $H_Q(2) = n \leq 5$, then $\p^Q(t)$ is rational.
\end{proposition}

\begin{proof}
We see that $\ell\ell(B)=2$, since $H_G(i)=1=H_A(i)+H_B(i)$ for $i\geq3$, and $\ell\ell(A)>\ell\ell(B)$. Thus, we have $$H_Q(s-1)=H_G(s-1)=H_A(s-1)=\edim(A)$$
where the last equality holds because $A$ is graded Gorenstein. So, $\edim(A)=1$ because $H_Q(s-1)=1$ as $s-1>3$.
Thus, by Corollary \ref{Loewy(B)=2}, $Q$ is a connected sum. If $n \leq 5$, the rationality of $\p^Q(t)$ follows from the previous corollary.
\end{proof}

\noindent
{\bf Acknowledgement}\\
We would like to thank L. L. Avramov for comments inspiring this work. We would also like to thank him,
A. A. Iarrobino and M. E. Rossi for insightful discussions and encouraging words along the way.

\end{document}